\documentclass[10.9pt,a4paper]{amsart}
\usepackage{a4wide}

\usepackage[utf8]{inputenc}
\usepackage[english]{babel}

\usepackage{amsfonts,amssymb,amsmath,pinlabel,array,hhline}
\usepackage[nobysame,alphabetic, initials]{amsrefs}

\usepackage[dvipsnames]{xcolor}
\usepackage{xcolor}
\usepackage[colorlinks,
    linkcolor={red!50!black},
    citecolor={blue!50!black},
    urlcolor={blue!80!black}]{hyperref}

\usepackage[position=b]{subcaption}

\usepackage{tikz}
\usepackage[all]{xy}
\usepackage{enumitem}
\makeatletter
\newcommand{\mylabel}[2]{#2\def\@currentlabel{#2}\label{#1}}
\usetikzlibrary{calc, matrix, arrows, cd}

\newcommand{\bsm}{\left(\begin{smallmatrix}}
\newcommand{\esm}{\end{smallmatrix}\right)}

\newtheorem{theorem}{Theorem}[section]

\theoremstyle{definition}
\newtheorem{definition}[theorem]{Definition}

\newtheorem{example}[theorem]{Example}

\newtheorem{remark}[theorem]{Remark}
\newtheorem{construction}[theorem]{Construction}
\newtheorem*{claim}{Claim}
\newtheorem*{claim*}{Claim}

\newcommand{\Z}{\mathbb{Z}}

\newcommand{\C}{\mathbb{C}}

\newcommand{\ks}{\operatorname{ks}}
\newcommand{\id}{\operatorname{id}}

\newcommand{\Aut}{\operatorname{Aut}}

\begin{document}
\title{Simple spheres in simply-connected $4$-manifolds}
\author{Anthony Conway}
\address{The University of Texas at Austin, Austin TX 78712}
\email{anthony.conway@austin.utexas.edu}
\begin{abstract}
These notes, which are based on three lectures delivered at the summer school ``Topological 4-manifolds" at CRM in~2025, discuss classifications of locally flat spheres in closed, simply-connected $4$-manifolds, with a focus on the case where the complements of the spheres have abelian fundamental group.
\end{abstract}

\def\subjclassname{\textup{2020} Mathematics Subject Classification}
\expandafter\let\csname subjclassname@1991\endcsname=\subjclassname
\subjclass{
57N35.
}
\maketitle

The study of surfaces in $4$-manifolds is currently undergoing fast-paced developments.
In order to provide an introduction to this topic in only three lectures,  these notes focus on knotted spheres.
Thus,  even though results on higher genus surfaces are occasionally referenced,  little mention is made of surfaces with boundary and of nonorientable surfaces.
The underlying belief is that, for an introduction to the topic,  working with knotted spheres suffices to get an impression of the results and techniques used in this area.

\medbreak

Lecture~\ref{sec:1} begins with some definitions and examples, Lecture~\ref{sec:2} is concerned with classifications of simple spheres, and Lecture~\ref{sec:3} concludes with realisation results.
Some exercises are listed at the end of the notes.
Since these lectures were part of a summer school devoted to topological~$4$-manifolds,  we will work exclusively in the topological category with locally flat embeddings.
Manifolds are assumed to be compact, connected and oriented.

\medbreak

Finally, an important disclaimer is in order.
These notes are neither meant to serve as a textbook nor as a survey of the area.
As a consequence,  during proof outlines,  references will be given less methodically and comprehensively than they would be in those settings.

\subsection*{Acknowledgments}

Heartfelt thanks go to the organisers of the summer school and conference ``Topological $4$-manifolds" for setting up a wonderful event and for the opportunity to talk about this topic.
This work was partially supported by the NSF grant DMS~2303674.

\section{Lecture one: Definitions, examples and invariants}
\label{sec:1}

During this first lecture, we introduce some definitions related to the study of knotted spheres, before moving on to examples and invariants.

\subsection{Definitions and examples}

Throughout these lectures,  embeddings will be understood to be locally flat.
In what follows, we fix a closed, simply-connected~$4$-manifold~$X$.
Homeomorphisms are assumed to be orientation-preserving.
We begin by fixing some terminology.

\begin{definition}
\label{def:Isotopy}
Two embedded spheres~$S_0 \subset X$ and~$S_1 \subset X$ are
 \emph{equivalent} if there is a homeomorphism~$\Phi \colon X \to X$ with~$\Phi(S_0)=S_1$
and \emph{isotopic} if, moreover,~$\Phi$ is isotopic to the identity.
\end{definition}

One of the main goals in the study of knotted spheres is to understand the set of isotopy classes of spheres in~$X$.
Potentially more manageable tasks include 
\begin{itemize}
\item classifying simple spheres in~$X$ up to isotopy (or equivalence),
\item classifying spheres in~$X$ up to concordance.
\end{itemize}
Here a sphere~$F \subset X$ is \emph{simple} if~$\pi_1(X \setminus S)$ is abelian and therefore cyclic (this is an exercise) and \emph{concordant} if there is an embedded~$S^2 \times [0,1] \cong C \subset X \times [0,1]$ with~$\partial C =S_0 \sqcup -S_1$, where $S_i \subset X \times \{ i\}$ for $i=0,1$.
Building on work of Kervaire on even-dimensional knots~\cite{Kervaire}, Sunukjian proved that homologous surfaces in $X$ are concordant~\cite{Sunukjian} (for recent work regarding concordance in more general~$4$-manifolds, we refer to~\cite{MillerConcordanceAnalogue,KlugMillerMoreGeneral}) so, in fact,  during these lectures, we will mostly focus on the isotopy classification of simple spheres.

\begin{example}
\label{ex:Examples2Knots}
We list examples of spheres in~$4$-manifolds.
\begin{itemize}
\item The \emph{unknot} $U \subset S^4$ refers to the usual embedding~$S^2 \subset S^3$ viewed as a submanifold of~$S^4$ by decomposing the latter as~$S^4=D^4 \cup_{S^3} D^4$.
The exterior $X_U$ of the unknot is homeomorphic to $S^1 \times D^3$ and the unknot is therefore simple: $\pi_1(X_U) \cong \pi_1(S^1 \times D^3)\cong \Z.$
Freedman proved that if $S \subset S^4$ has $\pi_1(X_S) \cong \Z$, then $S$ is isotopic to the unknot~\cite{FreedmanICM}.
This implies that in $S^4$, the only simple knot (up to isotopy) is the unknot.
\item We briefly review Zeeman's spinning construction~\cite{Zeeman} which, starting from a classical knot~$K \subset S^3$, constructs a $2$-knot $\tau^0(K)$ with $\pi_1(S^4 \setminus \tau^0(K))\cong \pi_1(S^3 \setminus K)$ (thus,  there are infinitely many nontrivial $2$-knots in $S^4$).
Writing $S^4=(D^3 \times S^1) \cup (S^2 \times D^2)$ and considering the tangle $K^\circ \subset D^3$ obtained by removing a small  unknotted interval from~$K$, the \emph{spun $2$-knot $\tau^0(K)$} refers to the $2$-knot
$$ (S^4,\tau^0(K))=((D^3 \times S^1) \cup (S^2 \times D^2),K^\circ \times S^1 \cup \partial K^\circ \times D^2).$$
This construction can be generalised to twist spun knots $\tau^n(K)$~\cite{Zeeman} (so that the above spinning occurs when $n=0$) and roll spun knots~\cite{LitherlandDeforming}.
We do not dwell on these constructions since, by Freedman, they do not produce nontrivial simple knots.
\item In~$\C P^2$,  the usual~$\C P^1 \subset \C P^2$ is an embedded sphere that represents a generator of~$H_2(\C P^2)$; this sphere is simple since its complement is simply-connected.
A spherical representative for twice a generator arises from the algebraic curve~$x^2 + y^2+z^2=0$; it can be verified that the complement of this sphere has $\pi_1=\Z_2$.
Reversing the orientation of these spheres  leads to representatives for the
classes $-1$ and $-2$.
A simple sphere representing~$0\in H_2(\C P^2)$ is obtained by considering the unknot $U \subset S^4$ in a small ball~$B^4 \subset \C P^2$.
Tristram proved that~$0,\pm 1,\pm 2 \in H_2(\C P^2) \cong \Z$ are the only classes with spherical representatives~\cite[page 264]{Tristram}.
\item Potentially interesting spheres can also be obtained by more ad-hoc methods. 
For example, in~$S^4$,  spheres can be obtained by taking the union of discs in~$D^4$ with boundary a common knot.
Diagrammatical methods include band unlink diagrams~\cite{SwentonCalculus,HughesKimMillerBanded} and broken surfaces diagrams, see e.g.~\cite{CarterKamadaSaito}.
\end{itemize}
\end{example}

\subsection{Invariants of embedded spheres}

A rapid calculation shows that the homology of the exterior~$X_S:=X\setminus \nu(S)$ only depends on $X$ and on the homology class of $S \subset X$.
On the other hand, the homotopy groups of $X_S$ are powerful invariants but are often difficult to calculate.
Other invariants of $2$-knots in $S^4$ include the Farber-Levine pairing and Casson-Gordon invariants, we refer to~\cite{ConwaySurvey} and to the references within for further information (as well as for a discussion of invariants of smooth $2$-knots).
The reason for which we do not dwell on these latter invariants is again that, by Freedman, the only simple knot in $S^4$ is the unknot.
Instead we move on to intersection forms on the homology of (branched) coverings of sphere exteriors.

\begin{construction}
\label{cons:Forms}
Let $S \subset X$ be a sphere of divisibility $d$, meaning that $d \geq 0$ is the largest integer dividing~$[S] \in H_2(X) \cong \Z^{b_2(X)}$.
As noted in the exercises, this implies that $H_1(X_S) \cong \Z_d$ and thus, if $S$ is simple,  that $\pi_1(X_S)\cong \Z_d$.
\begin{itemize}
\item If $S$ has divisibility $d \neq 0$,  then $H_1(X_S) \cong \Z_d$ and the intersection form~$Q_{\Sigma_d(S)}$ of the~$d$-fold branched cover $\Sigma_d(S)$ is an invariant of $S$.
In order to keep track of the action of the deck transformation group $\Z_d=\langle T \mid T^d=1\rangle$, it is also convenient to consider $H_2(\Sigma_d(S))$ as a $\Z[\Z_d]$-module and to consider the hermitian \emph{equivariant intersection form}
\begin{align*}
\lambda_{\Sigma_d(S)} \colon H_2(\Sigma_d(S)) \times H_2(\Sigma_d(S)) &\to \Z[\Z_d] \\
(x,y) &\mapsto \sum_{k=0}^{d-1} Q_{\Sigma_d(S)}(x,T^ky) T^{-k}.
\end{align*}
Here,  being \emph{hermitian} means that~$\lambda_{\Sigma_d(S)}(y,x)=\overline{\lambda_{\Sigma_d(S)}(x,y)}$, where the overline refers to replacing all instances of $T$ by $T^{-1}$.
As an exercise, the reader can calculate $H_i(\Sigma_d(S))$ as an abelian group and verify that, when $S$ is simple,  the pairings~$Q_{\Sigma_d(S)}$ and $\lambda_{\Sigma_d(S)} $ are nonsingular.
More challenging is the fact that if~$S$ is simple, then $H_2(\Sigma_d(S))$ is free as a $\Z[\Z_d]$-module~\cite[Theorem 3.4]{LeeWilczyOdd}.
An Euler characteristic calculation then shows that~$H_2(\Sigma_d(S)) \cong \Z[\Z_d]^{b_2(X)}.$
 \item If $S \subset X$ is nullhomologous, then $H_1(X_S)\cong \Z$,  and writing $\Z \cong \langle T \rangle$ for the deck transformation group of the infinite cyclic cover $X_S^\infty$ of $X_S$,  the \emph{equivariant intersection form} of $X_S$ refers to the hermitian form
 \begin{align*}
\lambda_{X_S} \colon H_2(X_S^\infty) \times H_2(X_S^\infty) &\to \Z[\Z] \\
(x,y) &\mapsto \sum_{k \in \Z} (x \cdot T^ky) T^{-k}.
\end{align*}
If $S \subset X$ is a $\Z$-sphere, then the module $H_2(X_S^\infty)$ is free of rank $b_2(X)$ as a $\Z[\Z]$-module~\cite[Lemma 3.2]{ConwayPowell}: $H_2(X_S^\infty) \cong \Z[\Z]^{b_2(X)}$.
One can additionally verify that~$\lambda_{X_S} $ is nonsingular~\cite[Lemma 3.2]{ConwayPowell}.
These facts feature as suggested exercises.
\end{itemize}
\end{construction}

We note that both $\lambda_{\Sigma_d(S)}$ and $\lambda_{X_S}$ \emph{augment} to the intersection form of $X$ (denoted~$Q_X$),  meaning that if one sets $T=1$ in matrices representing these forms,  then one obtains a matrix representing~$Q_X$.
More formally, this means that for $d\geq 0$, if one endows $\Z$ with the $\Z[\Z_d]$-module structure given by $Tx=x$ for $x \in \Z$, then $(H_2(X_S^\infty),\lambda_{X_S})\otimes_{\Z[\Z]} \Z \cong (H_2(X),Q_X) $ and for $d \neq 0$,  similarly,~$(H_2(\Sigma_d(S)),Q_{\Sigma_d(S)}) \otimes_{\Z[\Z_d]} \Z \cong (H_2(X),Q_X)$.

\begin{example}
We calculate the equivariant intersection form for $S \subset \C P^2$ a divisibility $d$ simple sphere.
Without loss of generality, we assume that $[S]=d \in \Z \cong H_2(\C P^2)$ is nonnegative,  since the equivariant intersection form is insensitive to reversing the orientation of the sphere.
\begin{enumerate}
\item If $d=1$, then $\Sigma_d(S)=\C P^2$ so the intersection form is represented by $(1)$.
\item If $d=2$, then $H_2(\Sigma_2(S)) \cong \Z[\Z_2]$ and since the equivariant intersection is hermitian, nonsingular and augments to $Q_{\C P^2}\cong (1)$,  we deduce that $\lambda:=\lambda_{\Sigma_2(S)}$ is either represented by~$(1)$ or~$(T)$.
A little work shows that $\Sigma_2(S)$ is spin (see~\cite[Lemma 4.5 (2)]{LeeWilczy} for a vastly more general statement) and thus, if $y \in H_2(\Sigma_2(S)) \cong \Z[\Z_2]$ is a $\Z[\Z_2]$-generator, then~$\lambda(y,y)=y \cdot y + (y \cdot Ty)T=2n+(y \cdot Ty)T$ for some $n \in \Z$.
Since we saw that~$\lambda(y,y)$ is either $1$ or $T$, we deduce that~$\lambda$ is represented by~$(T)$.
\item If $d=0$, then $H_2(X_S;\Z[\Z])\cong \Z[\Z]$ and since the equivariant intersection is hermitian, nonsingular and augments to $Q_{\C P^2}\cong (1)$ we deduce that it is represented by $(1)$.
\end{enumerate}
\end{example} 

\begin{remark}
When $S \subset X$ is a divisibility $d$ sphere with $d \neq 0$,  we will need a slight refinement of the equivariant intersection form.
Indeed, it will be helpful to consider~$\lambda_{\Sigma_d(S)}$ together with the (necessarily primitive if $S$ is simple) homology class $[\widetilde{S}] \in H_2(\Sigma_d(S))$ of the lift of $S$ to the branched cover.
An isometry of such \emph{pointed hermitian forms} $(P,\lambda,z)$ and $(P',\lambda',z')$ consists of an isometry~$(P,\lambda) \cong (P',z')$ of the underlying hermitian forms that takes $z$ to $z'$.
For brevity,  we will often say that $\lambda$ and $\lambda'$ are \emph{isometric as pointed hermitian forms}.
\end{remark}

Before delving further into the study of simple spheres,  we conclude with some remarks on the general study of knotted spheres in $4$-manifolds.
Arguably, one challenge in this topic is the relative lack of problems beyond those related to concordance and isotopy.
Some themes of interest include developing suitable notions of unknotting,  the study of ribbon $2$-knots,  and stable phenomena, i.e. studying properties of $2$-knots up to internal or external stabilisation.

\section{Lecture two: Simple spheres up to isotopy}
\label{sec:2}

During this lecture, we describe instances in which complete invariants of simple spheres are available.
The first result of this type is due to Freedman, who proved that~$\Z$-spheres in $S^4$ are unknotted~\cite{FreedmanICM}.
The next result describes what is known about simple spheres more generally.

\begin{theorem}
\label{thm:InvariantsSuffice}
Let $X$ be a closed,  simply-connected $4$-manifold.
\begin{enumerate}
\item Lee-Wilczy\'nski~\cite{LeeWilczyOdd,LeeWilczy}: For $d\neq 0$,  two homologous simple spheres~$S_0,S_1 \subset X$ of divisibility $d$ are equivalent if and only if their pointed equivariant intersection forms are isometric, i.e.~$\lambda_{\Sigma_d(S_0)} \cong \lambda_{\Sigma_d(S_1)}$ as pointed hermitian forms.
\item Two $\Z$-spheres $S_0,S_1 \subset X$ are equivalent if and only if their equivariant intersection forms are isometric, i.e.  $\lambda_{X_{S_0}} \cong \lambda_{X_{S_1}}$~\cite[Theorem 1.4]{ConwayPowell}.
\end{enumerate}
\end{theorem}

No such results appear to be available for nonabelian knot groups, but we briefly mention some partial results.
Hillman proved that ribbon~$2$-knots in~$S^4$ with knot group the Baumslag-Solitar~$BS(1,2)=\langle a,b \mid aba^{-1}=b^2\rangle$ are determined by the homotopy type of~$M_K$, the surgery on~$K$~\cite[Corollary on page 140]{Hillman2knotgroups}, and therefore (thanks to~\cite{HambletonKreckTeichner}) by the equivariant intersection form of the latter.
We also refer to~\cite[Theorem 2]{Hillman2knotgroups} for related results.

\begin{remark}
\label{rem:Cancellation}
We record a couple of remarks on Theorem~\ref{thm:InvariantsSuffice}.
\begin{itemize}
\item Lee-Wilczy\'nski~\cite{LeeWilczyOdd,LeeWilczy} describe cases where the existence of a pointed isometry~$\lambda_{\Sigma_d(S_0)} \cong \lambda_{\Sigma_d(S_1)}$ is automatic,  and where equivalence can be replaced by isotopy, e.g. when~$d=1$
and when~$X$ satisfies both~$b_2(X)>6$ and
\begin{equation}
\label{eq:UniquenessIneq}
 b_2(X) > \underset{0 \leq j <d}{\operatorname{max}} \ \Big| \sigma(X)-\frac{2j(d-j)}{d^2}x\cdot x \Big|.
 \end{equation}
Under certain additional hypotheses, Lee and Wilczy\'nski also obtain such  uniqueness results when~$3 \leq b_2(X) \leq 6$~\cite[Addendum 1]{LeeWilczy}; see also~\cite[Theorem~4.5]{HambletonKreck2} for the case when~$b_2(X)>|\sigma(X)|+2$ and~\eqref{eq:UniquenessIneq} holds.
These results rely on difficult algebra related to cancellation problems in the theory of quadratic forms.
The case $d=1$ was also obtained in~\cite[Theorem~G]{BoyerRealization} (for surfaces of arbitrary genus) by different methods.
We also note that the inequality in~\eqref{eq:UniquenessIneq} rules out $X$ being definite.
\item In the second statement,  work of Hambleton-Teichner~\cite{HambletonTeichner} shows that $\lambda_{X_{S_0}} \cong \lambda_{X_{S_1}}$ is automatic provided~$b_2(X) \geq \sigma(X)+6$, thus ensuring that $S_0$ and $S_1$ are necessary equivalent.
In this setting, Sunukjian claims that isotopy can be achieved~\cite[Theorem~7.2]{Sunukjian}.
We also note that the second statement of Theorem~\ref{thm:InvariantsSuffice} holds for closed $\Z$-surfaces of arbitrary genus~\cite[Theorem 1.4]{ConwayPowell}.
\end{itemize}
In both statements,  work of Quinn~\cite{Quinn} ensures that equivalence can be improved to isotopy provided there is an isometry $F$ such that the ``augmented isometry" $F(1) \colon H_2(X) \to H_2(X)$ is the identity.
We spell this out in the case $d=0$ (the case $d \neq 0$ is similar but with pointed isometries): if there is an isometry~$F \colon \lambda_{S_0} \cong  \lambda_{S_1}$ 
such that the augmented isometry:
$$F(1):=F \otimes \id_\Z \colon \overbrace{H_2(X;\Z[\Z]) \otimes_{\Z[\Z]} \Z}^{\cong H_2(X)} \to \overbrace{H_2(X;\Z[\Z])  \otimes_{\Z[\Z]} \Z}^{\cong H_2(X)}$$
is the identity, then $S_0$ and $S_1$ are isotopic.
Here,  Quinn's theorem states that if a homeomorphism~$X\to X$ induces the identity on $H_2$, then it is isotopic to the identity.
\end{remark}

Contrarily to what the previous remark might suggest,~$\Z$-spheres in a closed simply-connected~$4$-manifold need not be equivalent.
For example, it is known that there is a nontrivial~$\Z$-sphere in~$\#_4 \C P^2$~\cite[Theorem~C]{Torres}; this also follows by combining Theorem~\ref{thm:Realisation} with~\cite[Theorem~1]{HambletonTeichner}.
We note in passing that the situation for surfaces with boundary is very different.
For example, a knot can bound infinitely many $\Z$-discs in~$\C P^2 \setminus \operatorname{Int}(B^4)$ up to isotopy rel. boundary,  and the disc exteriors all have isometric equivariant intersection form~\cite{ConwayDaiMiller}.

\begin{example}
We have already seen that Freedman's theorem implies that, up to isotopy, there is a single simple sphere in $S^4$, namely the unknot.
Applying Theorem~\ref{thm:InvariantsSuffice} together with a little additional work shows that homologous simple spheres in $\C P^2$ are isotopic~\cite{ConwayOrson}.
In particular, combined with the work of Tristram that was mentioned in Example~\ref{ex:Examples2Knots}, this implies that up to isotopy, there are $5$ simple spheres in $\C P^2$.
In more detail, Theorem~\ref{thm:InvariantsSuffice} and Remark~\ref{rem:Cancellation} give the result for $d=\pm 1$; for $d=0$, Theorem~\ref{thm:InvariantsSuffice} implies that homologous spheres are equivalent and an additional argument from~\cite[Appendix]{ConwayOrson} establishes isotopy; for $d=\pm 2$, this is the main result of~\cite{ConwayOrson}, though this might also follow from Theorem~\ref{thm:InvariantsSuffice} provided one is able to pin down the \emph{pointed} isometry type of the equivariant intersection form.
\end{example}

We conclude with a proof outline of Theorem~\ref{thm:InvariantsSuffice}.

\begin{proof}[Proof outline of Theorem~\ref{thm:InvariantsSuffice}]
Implicit in Construction~\ref{cons:Forms} is the fact that the equivariant intersection form is an invariant of knotted spheres, i.e. that equivalent knotted spheres have isometric equivariant intersection forms.
This establishes one direction of (each assertion of) the theorem and we therefore focus on the converses.
We begin with a brief proof outline of both results.

When $d \neq 0$, the idea is to show that any pointed isometry~$F \colon \lambda_{\Sigma_d(S_0)} \cong \lambda_{\Sigma_d(S_1)}$ is realised by an equivariant homeomorphism $(\Sigma_d(S_0),\overline{\nu}(\widetilde{S}_0)) \to (\Sigma_d(S_1),\overline{\nu}(\widetilde{S}_1))$ taking $\widetilde{S}_0$ to $\widetilde{S}_1$ or,  equivalently (by passing to orbit spaces),  by a homeomorphism $(X,\overline{\nu}(S_0)) \to (X,\overline{\nu}(S_1))$ taking $S_0$ to $S_1$.
If such a homeomorphism realises $F$, then on second homology it induces
$$F(1) :=F\otimes \id_\Z \colon \overbrace{H_2(\Sigma_d(S_0)) \otimes_{\Z[\Z_d]} \Z}^{\cong H_2(X)} \to \overbrace{H_2(\Sigma_d(S_1)) \otimes_{\Z[\Z_d]} \Z}^{\cong H_2(X)}.$$
To obtain an ambient isotopy instead of merely an equivalence of surfaces,  the aim is to show that~$F$ can be chosen so that~$F(1)=\id_{H_2(X)}$ and to then appeal to Quinn~\cite{Quinn}.

When~$d=0$, the idea is to show that an isomety~$F \colon \lambda_{X_{S_0}} \cong \lambda_{X_{S_1}}$ can be realised by a homeomorphism~$X_{S_0} \to X_{S_1}$ whose restriction to~$\partial X_{S_0} \to \partial X_{S_1}$ extends to a homeomorphism~$\overline{\nu}(S_0) \to \overline{\nu}(S_1)$ taking $S_0$ to $S_1$.
It then follows that this homeomorphism gives rise to the required homeomorphism~$(X,S_0) \to (X,S_1)$; again appealing to Quinn~\cite{Quinn}, one sees that if~$F(1)=\id_{H_2(X)}$, then the spheres are isotopic.

Thus, in both cases, the strategy involves:
\begin{enumerate}
\item Choosing a homeomorphism $f \colon \overline{\nu}(S_0) \to \overline{\nu}(S_1)$ taking $S_0$ to $S_1$.
\item Showing that $f|_{\partial}$ extends to a homotopy equivalence $f \colon X_{S_0} \to X_{S_1}$ that realises $F$.
\item Using surgery theory to find a homeomorphism $\Phi \colon X_{S_0} \to X_{S_1}$ with $\Phi|_{\partial}=f|_{\partial}$ and that still realises $F$.
\end{enumerate}
When $d=0$,  the strategy above doesn't follow the original proof from~\cite{ConwayPowell} but is made possible by upcoming joint work with Daniel Kasprowski on the homotopy classification of $4$-manifolds with a given boundary.

The first of these three steps isn't overly challenging and the second is an intricate obstruction theoretic argument with roots in~\cite[Proof of Theorem 1.1]{HambletonKreck}.
Following~\cite[Proof of Proposition 4.6]{LeeWilczy}, we therefore focus on the third.\footnote{This was also a way to connect with the other minicourses that had been building up surgery theory.}

We begin with the case where $d\neq 0$.
Set $\pi:=\Z_d$ and~$X_i:=X_{S_i}$ for brevity.
We assume that there is a homotopy equivalence $f \colon X_0 \to X_1$ that extends to a homeomorphism $\overline{\nu}(S_0) \to \overline{\nu}(S_1)$ taking $S_0$ to $S_1$ and whose extension to $\Sigma_d(S_0) \to \Sigma_d(S_1)$ realises the isometry~$F$. 
It is known that $f$ is automatically a simple homotopy equivalence (this is because the Whitehead torsion of $f$ lies in $SK_1(\Z[\pi])$ and the latter vanishes for cylic groups),  thus ensuring that $f$ defines an element in the simple structure set $\mathcal{S}^s(X,\partial X)$.
Consider the following portion of the surgery sequence (which is exact since $\Z_d$ is good~\cite{FreedmanQuinn}; see also~\cite[Chapter 22]{DET}):
$$\ldots \xrightarrow{\sigma_5} L_5^s(\Z[\pi]) \dashrightarrow \mathcal{S}^s(X_1,\partial X_1) \xrightarrow{\eta} \mathcal{N}(X_1,\partial X_1) \xrightarrow{\sigma_4} L_4^s(\Z[\pi]).$$
The plan is to show that there is a simple homotopy equivalence $h \colon X_1 \to X_1$ that restricts to the identity on $\partial X_1$,  induces the identity on $\pi_1$ and $\pi_2$, and is such that $h \circ f$ is homotopic to a homeomorphism.

In other words,  we consider the group~$\Aut^s_{\id}(X_1)$ of rel. boundary homotopy classes of such simple homotopy equivalences, the action~$\Aut^s_{\id}(X_1) \curvearrowright \mathcal{S}^s(X_1,\partial X_1)$ by composition, and aim to show that~$\mathcal{S}^s_{\ks}(X_1,\partial X_1)/\Aut^s_{\id}(X_1)$ is trivial, where 
$$\mathcal{S}^s_{\ks}(X_1,\partial X_1)=\{ [f \colon (M,\partial M) \to (X_1,\partial X_1)] \in \mathcal{S}^s(X_1,\partial X_1)  \mid \ks(M)=\ks(X_1) \rbrace.$$
For any group $\pi$,  it is known that~$\mathcal{N}(X_1,\partial X_1)\cong H^4(X_1,\partial X_1) \oplus H^2(X_1,\partial X_1;\Z_2)$; see e.g.~\cite[Section 22.3.2]{DET}.
When~$\pi=\Z_d$ is finite cyclic,  it is additionally known that~$L_5^s(\Z[\pi])=0$ and~$\ker(\sigma_4) \cong H^2(X_1,\partial X_1;\Z_2)$  so that~$\eta$ induces a bijection~$\eta \colon \mathcal{S}^s(X_1,\partial X_1) \xrightarrow{\cong} H^2(X_1,\partial X_1;\Z_2)$.
References for these deep results can be found in~\cite[Proof of Proposition 4.6]{LeeWilczy} as can the fact
that~$\eta$ restricts to a bijection
$$  \eta| \colon  \mathcal{S}^s_{\ks}(X_1,\partial X_1) \xrightarrow{\cong} \{ x \in H^2(X_1,\partial X_1;\Z_2) \mid x \cup w_2(X_1)=0\rbrace.$$
A construction that goes back to Novikov (see~\cite{CochranHabegger} for the details and to~\cite[Lemmas 4.14 and 4.15]{LeeWilczy} for the case~$\pi=\Z/2k$) shows that when $\pi$ is cyclic,  every~$x$ in the target of this map can be realised as~$\eta(h)$ for some 
homotopy equivalence~$h \in \Aut^s_{\id}(X_1)$, see e.g.~\cite[Lemma~4.10]{LeeWilczy}.

Applying this with~$x=\eta(f)$, further properties of this ``Novikov pinching" construction show that~$\eta(h \circ f)=0$~\cite[Lemma 4.9]{LeeWilczy}.
Since $\eta|$ is a bijection, ~$h \circ f \sim \id \in \mathcal{S}^s(X_1,\partial X_1)$, meaning that~$h \circ f$ is homotopic rel. boundary to the required homeomorphism~$\Phi$.
For~$\pi=\Z$, the argument is similar but,  this time,  one uses that~$\operatorname{Wh}(\pi)=0$
and that~$\sigma_5$ is surjective.
\end{proof}

\section{Lecture three: Enumerating simple spheres}
\label{sec:3}

The previous lecture stated criteria ensuring that homologous $\Z_d$-spheres are isotopic.
In order to enumerate simple spheres in a given simply-connected closed $4$-manifold $X$, it therefore remains to know which homology classes can be represented by simple spheres and which (pointed) hermitian forms arise as (pointed) equivariant intersection forms.
The following theorem summarises what is known on the subject.

\begin{theorem}
\label{thm:Realisation}
~
\begin{itemize}
\item Lee-Wilczy\'nski~\cite{LeeWilczyOdd,LeeWilczy}: a nonzero divisibility $d$ class $x \in H_2(X)$ is represented by a simple sphere if and only if 
\begin{equation}
\label{eq:ExistenceLW}
 b_2(X) \geq  \underset{0 \leq j <d}{\operatorname{max}} \ \Big| \sigma(X)-\frac{2j(d-j)}{d^2}x\cdot x \Big|
 \end{equation}
and, when $x$ is characteristic,  $x$ additionally satisfies
\begin{equation}
\label{eq:ExistenceCharacteristic}
\ks(X)=\frac{1}{8}(\sigma(X)-x\cdot x).
\end{equation}
\item Given a nonsingular hermitian form $\lambda(t)$ over $\Z[\Z]$,  there exists a $\Z$-sphere $S \subset X$ with equivariant intersection form $\lambda(t)$ if and only if~$\lambda(1) \cong Q_X$~\cite{ConwayPiccirilloPowell}.
\end{itemize}
\end{theorem}

Here,  a class $x \in H_2(X)$ is \emph{characteristic} if $x \cdot a=a \cdot a$ mod $2$ for every $a \in H_2(X)$ and is \emph{ordinary} otherwise.
For~$x$ characteristic,~$\sigma(X)-x \cdot x$ is divisible by~$8$ because the intersection form of $X$ is nonsingular; see e.g.~\cite[Proposition~1.2.20]{GompfStipsicz}.

\begin{remark}
We record a couple of remarks on Theorem~\ref{thm:Realisation}.
\begin{itemize}
\item Lee and Wilczy\'nski first proved their theorem for $d$ odd~\cite{LeeWilczyOdd}, whereas the result for $d$ even was obtained independently by Lee-Wilczy\'nski~\cite{LeeWilczy} and Hambleton-Kreck~\cite[Theorem 2]{HambletonKreck2}.
The statement for $d=1$ was also obtained by Boyer in~\cite[Theorem G]{BoyerRealization} for surfaces of arbitrary genus.
\item  The form obtained by extending $Q_X$ to a form over $\Z[\Z]$ is the archetypal example of a form as in the second statement of Theorem~\ref{thm:Realisation}.
In fact,  when $b_2(X) \geq |\sigma(X)|+6$, by~\cite[Theorem 2]{HambletonTeichner},  it is the only such form.
\end{itemize}
\end{remark}

\begin{remark}
The combination of Theorems~\ref{thm:InvariantsSuffice} and~\ref{thm:Realisation} shows that if a primitive class~$x \in H_2(X)$ is representable by a sphere, then this sphere is unique up to isotopy.
\end{remark}

Unfortunately,  due to the discrepancy between the inequalities~\eqref{eq:ExistenceLW} and~\eqref{eq:UniquenessIneq}, the combination of these theorems is often insufficient to enumerate all simple spheres in a given~$X$: no uniqueness statement is available for nonzero divisibility $d$ classes $x \in H_2(X)$
that are spherically representable but for which $ b_2(X) = \underset{0 \leq j <d}{\operatorname{max}} \ \Big| \sigma(X)-\frac{2j(d-j)}{d^2}x\cdot x \Big|$.

\begin{example}
\label{ex:Enumeration}
We apply Theorems~\ref{thm:InvariantsSuffice} and~\ref{thm:Realisation} to argue that there are infinitely many isotopy classes of simple spheres in~$X=(S^2 \times S^2)\# (S^2 \times S^2)$ (note~$b_2(X) >|\sigma(X)|+2$),  one for each spherically representable homology class, except possibly in the nullhomologous class.
For existence,  a calculation shows that a nonzero class~$(a,b,c,e) \in \Z^4 \cong H_2(X)$ satisfies~\eqref{eq:ExistenceLW} if and only if it is primitive or~$(ab+ce)=0$ (in which case the strict inequality~\eqref{eq:UniquenessIneq} actually holds),  that~$(a,b,c,e)$ is characteristic if and only if~$a,b,c,e$ are all even, 
and that such characteristic classes satisfy~\eqref{eq:ExistenceCharacteristic} if and only if~$ab+ce$ is even.
Thus, by Theorem~\ref{thm:Realisation}, the nonzero spherically realisable classes are the primitive ones as well as those with~$ab+ce=0$; in both cases the spheres are unique thanks to Theorem~\ref{thm:InvariantsSuffice}.
In the nullhomologous case,  the combination of Theorem~\ref{thm:InvariantsSuffice} and Theorem~\ref{thm:Realisation}
implies that there is a unique~$\Z$-sphere for every nonsingular hermitian form~$\lambda(t)$ with~$\lambda(1) \cong \bsm 0& 1 \\ 1&0 \esm^{\oplus 2}$.
It is unknown how many such forms exist.

In summary, isotopy classes of simple spheres in $X$ correspond bijectively to 
$$
 \left\{ x \in \Z^4 \text{ primitive} \right\}
 \sqcup \bigsqcup_{d \geq 2}   \left\{ d \bsm a \\ b \\ c \\ e\esm \ \Big| \ ab+ce=0  \right\}
 \sqcup \frac{\left\{ \lambda(t) \text{ nonsing.  hermitian} \mid \lambda(1) \cong \bsm 0&1  \\ 1&0 \esm^{\oplus 2} \right\}}
 {\lambda_0(t) \cong \lambda_1(t) \text{ via $F$ such that $F(1)=\id_{H_2(X)}$} }.
$$

\end{example}

We conclude with a proof outline of Theorem~\ref{thm:Realisation}, focusing on the sufficiency of the conditions.

\begin{proof}[Proof outline of Theorem~\ref{thm:Realisation} when $d =0$.]
We begin with the case of $\Z$-spheres.
We direct the reader to~\cite[Section 5]{ConwayPiccirilloPowell} for details and for references to facts mentioned without proof during this outline.
The necessity of the condition $\lambda(1) \cong Q_X$ was already alluded to below Construction~\ref{cons:Forms}.
For sufficiency,  the idea of the proof is to build a manifold that ressembles a sphere exterior and to then deduce the existence of the sphere.
Attach~$2$-handles equivariantly to the top of~$(S^2 \times  S^1) \times [0,1]$ according to the equivariant linking and framing dictated by~$\lambda$ resulting in a cobordism~$(W,S^2\times S^1,Y)$.
Argue that~$H_1(Y^\infty)=0.$
A surgery theoretic argument analogous to the one used for proving that Alexander polynonial one knots are slice shows that every such~$3$-manifold bounds a $4$-manifold~$B$ with~$B \simeq S^1$.
It follows that~$M:=W \cup_Y B$ is a~$4$-manifold with~$\pi_1 \cong \Z$, boundary~$S^1 \times S^2$, and equivariant intersection form~$\lambda_M \cong \lambda$.
In fact,~$\pi_1(S^1 \times S^2) \to \pi_1(M)$ is an isomorphism.
In the odd case, taking the star partner of $M$ to change its Kirby-Siebenmann invariant of $M$ if necessary,  one can assume that~$\ks(M)=\ks(X)$.
Form~$X':=M\cup (S^2 \times D^2)$ and consider the~$\Z$-sphere~$S:=S^2 \times \{0 \}.$

It remains to verify that~$X' \cong X$ as it will then follow that~$X$ contains a~$\Z$-sphere, namely the image of~$S$ under this homeomorphism.
First,  a van Kampen argument shows that~$X'$ is simply-connected.
Next, using the hypothesis on~$\lambda$, we obtain the sequence of isometries
$$Q_{X'} \cong \lambda_{X'}(1) \cong \lambda_M(1) \cong \lambda(1) \cong Q_X.$$
In the odd case,  the additivity of the Kirby-Siebenman invariant gives~$\ks(X')=\ks(M)=\ks(X)$.
Freedman's work now implies that~$X \cong X'$~\cite{Freedman}.
This concludes the construction of a~$\Z$-sphere~$S\subset X' \cong X$ whose exterior has equivariant intersection form~$\lambda$.
\end{proof}

\begin{proof}[Proof outline of Theorem~\ref{thm:Realisation} when $d \neq 0$.]
The necessity of the conditions in Lee and Wilczy\'nski's theorem can be verified using the work of Rochlin~\cite{Rochlin} (see also~\cite{FreedmanKirby} for~\eqref{eq:ExistenceCharacteristic}).
For sufficiency, the idea of the proof is to first
 show that the condition on the Arf invariant guarantees that,  for some~$k \geq 0$, the class~$x \oplus 0$ is represented by a simple sphere~$S \subset X \#_k S^2 \times S^2$~\cite[Theorem~2.1]{LeeWilczyOdd},
and to then use the main hypothesis of the theorem (namely~\eqref{eq:ExistenceLW}) to show that, informally speaking,  it is possible to surger away the~$S^2 \times S^2$-summands without damaging the sphere.

We elaborate on the first step, referring to~\cite[proof of Theorem~2.1]{LeeWilczyOdd} for further details.
Represent~$x \in H_2(X)$ by a surface~$F \subset X$.
The goal is to show that,  assuming~\eqref{eq:ExistenceCharacteristic} in the characteristic case,  it is possible to ambiently surger~$F$ to a sphere at the price of stabilising~$X$.
Consider the task of performing ambient surgery on an embedded loop~$\gamma \subset F$.
This requires the existence of an embedded disc~$D \subset X$ with~$\partial D=\gamma=D \cap F$ and such that disc framing on~$\gamma$ agrees with the surface-induced framing on it; see e.g.~\cite[page~508]{Scorpan} for details and illustrations.
Write~$W(D)\in \Z$ for the obstruction to this framing condition being realised. 
\begin{claim}
If~$W(D)$ and~$D \cdot F$ have the same parity, then ambient surgery along~$\gamma$ is possible in a stabilisation of $X$.
\end{claim}
\begin{proof}
Performing a so-called \emph{boundary twist} on~$D$ (see e.g.~\cite[p.~15]{FreedmanQuinn}) changes both quantities by~$\pm 1$ so that, after sufficiently many such operations, it is possible to assume that~$D \cdot F=0$ and that~$W(D)$ is even.
It is then possible to change~$W(D)$ by any even number (without affecting~$D \cdot F)$ at the price of stabilisations; see e.g.~\cite[page 93]{FreedmanKirby} or~\cite[page 150]{Scorpan}.
Thus we can assume that both~$W(D)$ and~$D \cdot F$ vanish.
The Whitney trick can be performed stably (see e.g.~\cite[page 150]{Scorpan}), thus making it possible to arrange that~$D \cap F=\gamma$, as needed.
\end{proof}
It therefore remains to find an embedded disc~$D$ such that~$W(D)+D\cdot F$ is even.
It turns out that if the class~$x$ is ordinary, this is always possible, whereas in the characteristic case, the quantity~$q_F([\gamma]):=W(D)+D\cdot F$ leads to a genuine obstruction.

Thus, in the ordinary case,  it is possible to iteratively stably ambiently surger the curves in a symplectic basis of~$H_1(F)$ leading to a sphere $S$, whereas it turns out that in the characteristic case, this is also possible provided~$0=\operatorname{Arf}(q_F)=\ks(X) -1/8(\sigma(X)-x \cdot x)$.
Here,  $\operatorname{Arf}(q_F)$ denotes the Arf invariant of the quadratic form $q_F$ (it is far from obvious that $[\gamma]\mapsto q_F([\gamma])$ determines a quadratic form~\cite{FreedmanKirby}) and the second equality is due to Rochlin; see also~\cite{FreedmanKirby}.
Additional stabilisations make it possible to arrange for~$x \oplus 0$ to be stably representable by a \emph{simple} sphere (surger push offs of curves representing the commutators of the fundamental group of the sphere complement).
This concludes our outline of the first step of the proof.

\medbreak

Finally, we describe the idea underlying the second step of the proof; see~\cite{LeeWilczyOdd,LeeWilczy} for details.
The first step provides a simple sphere~$S \subset X \#^k S^2 \times S^2$ representing~$x\oplus 0$ for some~$k \geq 0$.
Arduous algebra coupled with~\eqref{eq:ExistenceLW} ensures the existence of a pointed isometry~$\lambda_{\Sigma_d(S)} \cong \lambda' \oplus \bsm 0&1 \\ 1&0 \esm^{\oplus k}$, where the hyperbolic form is pointed by the zero element, and~$\lambda'(1) \cong Q_X$.
Informally, the goal is now to surger away the hyperbolic summand without damaging the sphere $S$.
Writing~$X_S^d$ for the~$d$-fold cover of~$X_S$ so that~$\Sigma_d(S)=X_S^d \cup \overline{\nu}(S)$,  a Mayer-Vietoris argument shows that the~$0 \oplus \Lambda^{2k}$-summand lives in~$H_2(X_S^d) \cong \pi_2(X_S^d) \cong \pi_2(X_S)$.
Since $\pi_1(X_S)\cong \Z_d$ is good, Freedman's sphere embedding theorem then ensures the existence of~$k$ pairs of framed transverse spheres in~$X_S$~\cite{Freedman}; see also~\cite[page 292]{DET}.
Surgering~$X_S$ along these framed spheres produces a~$4$-manifold~$M$ with~$\pi_1(M)\cong \Z_d$.
This way,~$X':=M \cup \overline{\nu}(S)$ is simply-connected,  contains a simple sphere~$S'$ and,  since short arguments show that $\ks(X')=\ks(M)=\ks(X_S)=\ks(X)$ and that there is a pointed isometry~$Q_{X'} \cong \lambda_{\Sigma_d(S')}(1) \cong \lambda'(1) \cong Q_X$,  the work of Freedman~\cite{Freedman} implies that there is a homeomorphism~$X'\cong X$ taking~$[S']$ to~$[S]$, as required.
\end{proof}

\section*{Exercises}
Fix a closed, connected, simply-connected $4$-manifold $X$.
\begin{enumerate}
\item Calculate the homology of sphere exteriors in~$X$; deduce that (a) if $\pi_1(X \setminus S)$ is abelian, then it is cyclic and (b) if $S$ additionally has divisibility $d$, then $\pi_1(X \setminus S) \cong \Z_d$.
A little harder: calculate the $\Z[\Z]$-homology of a $\Z$-sphere exterior (i.e. $H_i(X_S^\infty)$ as a $\Z[\Z]$-module).
Along the way,  describe the boundary of the exterior of a sphere in $X$.
\item Prove that the intersection form of the branched cover of a simple sphere is nonsingular;
deduce that the equivariant intersection form of the branched cover is nonsingular.
A little harder: prove that the equivariant intersection form of a~$\Z$-sphere exterior is nonsingular.
\item Think through the proofs from lectures~$2$ and~$3$ and list difficulties that arise when trying to prove the results for surfaces of arbitrary genus. 
\item Prove that for every nonsingular symmetric bilinear form~$Q$ over $\Z$, there is a closed simply-connected~$4$-manifold with intersection form~$Q$ (use without proof Freedman's result that integer homology~$3$-spheres bound contractible manifolds).
Understand the main steps of Boyer's argument that, given a nondegenerate integral symmetric form~$Q$ and a closed~$3$-manifold~$Y$, there exists a~$4$-manifold~$X$ with boundary~$\partial X=Y$ and~$Q_X \cong Q$~\cite[Section 2, page 31 onwards]{BoyerRealization}; focus on the case where~$Y$ is a rational homology~$3$-sphere.
Think through the~$\Z[\Z]$ analogue hinted at during the proof of the first item of Theorem~\ref{thm:Realisation}.
\item Bonus/Challenge questions:
\begin{enumerate}
\item Attempt to determine the pointed hermitian form of a simple sphere $S \subset \C P^2$ representing the class $2 \in \Z \cong H_2(\C P^2)$.
\item Attempt to conclude Example~\ref{ex:Enumeration}, i.e. is every nonsingular size~$4$ matrix~$A(t)$ with~$A(1) \cong \bsm 0& 1 \\ 1&0 \esm^{\oplus 2}$ necessarily congruent over~$\Z[\Z]$ to~$\bsm 0& 1 \\ 1&0 \esm^{\oplus 2}$?
\item Attempt to establish the analogue of Example~\ref{ex:Enumeration} for $X=\#^3 S^2 \times S^2$ (this time, in the nullhomologous case, since $b_2(X) \geq |\sigma(X)| +6$, the $\Z[\Z]$-equivariant intersection form is determined).
How about $\#^n S^2 \times S^2$?
\item More generally,  attempt to classify simple spheres in $4$-manifolds with small Betti numbers (other than $S^4$ and $\C P^2$).
\end{enumerate}
\end{enumerate}

\def\MR#1{}
\bibliography{BiblioMontreal}
\end{document}